\documentclass[12pt,leqno]{amsart}
\usepackage{amsmath}
\usepackage{amssymb}

\def\underset#1#2{{\mathrel{\mathop {{}_{} {#2}}\limits_{{#1}_{}}}}}
\def\upplim_#1{\underset{#1}{\overline\lim}\;}
\def\lowlim_#1{\underset{#1}{\underline\lim}\;}
\newtheorem{definition}[equation]{Definition}
\newtheorem{claim}[equation]{\indent{\it Claim}\rm }

\newtheorem{lemma}[equation]{Lemma}

\newtheorem{theorem}[equation]{Theorem}

\newcommand{\C}{{\mathbb{C}}}

\newcommand{\A}{{\mathbb{A}}}

\newcommand{\Z}{\mathbb{Z}}

\numberwithin{equation}{section}

\title[Meromorphic functions on annuli sharing finite sets]{Meromorphic functions on annuli sharing finite sets with truncated multiplicity} 

\date{ }
\author{Si Duc Quang}

\begin{document}

\begin{abstract} 
The purpose of this paper has twofold. The first is to establish a second main theorem for meromorphic functions on annuli and meromorphic function targets (may not be small functions) with truncated counting functions (truncation level 1) and with a detailed estimate for the error term. The second is to show that if the polynomial 
$$P_S(w)=(w-a_1)\cdots (w-a_q)$$
is a uniqueness polynomial for admissible meromorphic functions on an annulus $\A(R_0)$ such that $P'_S(w)$ has exactly $k$ distinct zeros and $q>\frac{(5k+7)\ell}{2\ell-175}$, then the set $S=\{a_1,\ldots,a_q\}$ is a finite range set with truncation level $\ell$ for admissible meromorphic functions on $\A(R_0)$. This result extends the previous result on the finite range set (with truncation level $\ell=\infty$) for holomorphic functions on $\C$ of H. Fujimoto.
\end{abstract}

\maketitle

\renewcommand{\thefootnote}{\empty}

\footnote{2010 \emph{Mathematics Subject Classification}: Primary 30D35.}

\footnote{\emph{Key words and phrases}: Meromorphic function, Nevanlinna theory, annulus, finite range set.}

\section{Introduction}

A nonconstant monic polynomial $P(w)$ is called a uniqueness polynomial for meromorphic functions (or holomorphic functions) if, for any nonconstant
meromorphic functions (or holomorphic functions) $f$ and $g$ on $\C$, the equation $P(f)=cP(g)$ implies $f=g$, where $c$ is a nonzero constant which possibly depends on $f$ and $g$.
A finite subset $S$ of $\C$ is said to be a unique range set for meromorphic functions (resp. holomorphic functions) if $f^*(S)=g^*(S)$ implies $f=g$ for arbitrary nonconstant meromorphic functions (resp. holomorphic functions) $f$ and $g$ on $\C$, where $f^*(S)$ and $g^*(S)$ denote the pull-back divisors of $S$. For $S=\{a_1,\ldots,a_q\}$, we consider the polynomial
\begin{align}\label{1.1}
P_S(w)=(w-a_1)\cdots(w-a_q).
\end{align}
Hence, if $S$ is a unique range set for meromorphic functions (resp. holomorphic functions) then $P_S(w)$ is a uniqueness polynomial for meromorphic functions (resp. holomorphic functions).

In \cite{Fu07}, Fujimoto extended the notion of the unique range set to the following.

\vskip0.2cm
\noindent
{\bf Definition A.}\ \textit{A finite subset $S$ of $\C$ is called a finite range set for holomorphic functions if, for any given nonconstant holomorphic function $g$, there exist only finitely many nonconstant holomorphic functions $f$ such that $f^*(S)=g^*(S)$.}

With the above definitions, Fujimoto proved a sufficient condition for a set to be a finite range set for holomorphic functions on $\C$ as follows.

\vskip0.2cm
\noindent
{\bf Theorem B.}\ \textit{Take a finite set $S=\{a_1,\ldots,a_q\}$ and assume that, for the polynomial $P_S(\omega)$ defined by (\ref{1.1}), $P'_S(w)$ has exactly $k$ distinct zeros. If $P_S(w)$ is a uniqueness polynomial for holomorphic functions and $q>k+2$, then $S$ is a finite range set for holomorphic functions. More precisely, for an arbitrarily given nonconstant holomorphic function $g$, there exist at most $\frac{2q-2}{q-k-2}$ holomorphic functions $f$ such that $f^*(S)=g^*(S)$.}

Our purpose in this paper is to extend the result of Fujimoto to the case where the meromorphic functions (not only holomorphic) are defined only on annuli and the condition $f^*(S)=g^*(S)$ is replaced by a more general condition that $f$ and $g$ share the finite set $S$ with multiplicity truncated by a certain level. In order to state our result, we give the following definition.

For $0<R_0\le +\infty$, we set an annulus $\mathbb A(R_0)=\{z\in\C; \frac{1}{R_0}<|z|<R_0\}$. 
\begin{definition}\label{1.2}
Let $S=\{a_1,\ldots,a_q\}$ be a set of distinct values in $\C$ and $\ell$ be a positive integer (may be $\ell=+\infty$). Two admissible meromorphic functions $f$ and $g$ on $\A(R_0)$ are said to share $S$ with multiplicity truncated to level $\ell$ if 
$$\sum_{i=1}^q\min\{\ell,\nu_{f-a_i}\}=\sum_{i=1}^q\min\{\ell,\nu_{g-a_i}\}.$$
The set $S$ is said to be a finite range set with the truncation level $\ell$ for admissible meromorphic functions on $\A(R_0)$ if for an arbitrarily given admissible meromorphic function $g$ on $\A(R_0)$, there exist finitely many admissible meromorphic functions $f$ on $\A(R_0)$ only such that $f$ and $g$ share $S$ with multiplicity truncated to level $\ell$.
\end{definition}
Here, a meromorphic function $f$ on $\A(R_0)$ is said to be admissible if it satisfies
$$ \underset{r\longrightarrow +\infty}{\mathrm{limsup}}\dfrac{T_0(r,f)}{\log r}=+\infty\text{ in the case }R_0=+\infty $$
or
$$   \underset{r\longrightarrow R_0^-}{\mathrm{limsup}}\dfrac{T_0(r,f)}{-\log (R_0-r)}=+\infty\text{ in the case }1< R_0<+\infty,$$
where $T_0(r,f)$ denotes the characteristic function of $f$, which is defined in Section 2.
Hence, the condition that $f$ and $g$ share $S$ with multiplicity truncated to level $\ell=+\infty$ is equivalent to the condition $f^*(S)=g^*(S)$ in the sense of divisors. The set $S$ is said to be a finite range set for admissible meromorphic functions on $\A(R_0)$ if it is a finite range set with the truncation level $+\infty$ for admissible meromorphic functions on $\A(R_0)$.

Our main result is stated as follows.

\begin{theorem}\label{1.3}
Take a finite set $S=\{a_1,\ldots,a_q\}$ of $q$ distinct values in $\C$ and assume that, for the polynomial $P_S(\omega)$ defined by (\ref{1.1}), $P'_S(w)$ has exactly $k$ distinct zeros. Let $\ell$ be a positive integer (may be $\ell=+\infty$). If $P_S(w)$ is a uniqueness polynomial for admissible meromorphic functions on an annulus $\A(R_0)$ and $q>\frac{(5k+7)\ell}{2\ell-175}$, then $S$ is a finite range set with truncation $\ell$ for admissible meromorphic functions on $\A(R_0)$. More precisely, for an arbitrarily given admissible meromorphic function $g$ on $\A(R_0)$, there exist at most four admissible meromorphic functions $f$ on $\A(R_0)$ such that $f$ and $g$ share the set $S$ with multiplicity truncated to level $\ell$.
\end{theorem}

Here, we would like emphasize that the proof of Fujimoto for Theorem B is based on the sharp second main theorem for meromorphic functions and small functions on $\C$ with counting functions truncated to level $1$ of Yamanoi \cite[Theorem 1]{Yam}. By the assumption of Theorem B, Fujimoto assumed that $P_S(f_i)=e^{\varphi_i}P_S(g)$ with an holomorphic function $\varphi_i$ for each holomorphic function $f_i$ such that $f_i^*(S)=g^*(S)$ and $\varphi'_i$ is a small function, and then he applied the second main theorem of Yamanoi. However, in Theorem \ref{1.3}, the such holomorphic function $\varphi_i$ does not exist. Therefore we have to establish a new second main theorem for meromorphic functions with meromorphic function targets (may not be small) on annuli with a more detailed error term estimate. Namely, we will prove the following second main theorem.
 
\begin{theorem}\label{1.4}
Let $g$ be a nonconstant meromorphic function on $\A (R_0)$. Let $a_1,\dots ,a_q$ be $q\ (q\ge 5)$ distinct meromorphic functions (may be equal to $\infty$). We have
$$\frac{2q}{5}T_0(r,g)\le \sum_{i=1}^q\overline N_0(r,\nu^0_{g-a_i})+35\sum_{i=1}^qT_0(r,a_i)+S(r)\ (1\le r< R_0),$$
where $S(r)=S_g(r)+\sum_{i=1}^qS_{a_i}(r)$ (here we set $T_0(r,a_i)=0$ if $a_i\equiv\infty$).
\end{theorem}
Here, $\overline N_0(r,\nu)$ denotes the counting function without multiplicity of the divisor $\nu$ and the quantity $S_f(r)$ is the small term with respect to the meromorphic function $f$. These notions are defined in Section 2.

We note that, the weak version of this theorem for the case of small functions $a_i\ (1\le i\le 5)$ is initially stated in \cite{QAG} without proving (by repeating the method in \cite{YC}). Then the proof in this paper will also provide a detailed proof for that weak version, although the proof is based on the method of the proofs of Lemma 3.1 and Lemma 3.2 in \cite{YC}.

Finally, we remark that if $\ell=+\infty$ then the assumption of Theorem \ref{1.3} is fulfilled with $q>\frac{5k+7}{2}$. Unfortunately, this number is bigger than the number $k+2$ in Theorem B of Fujimoto. Therefore, ``\textit{how to find the optimal number $k$ in our situation}'' is still an open interesting question.

\section{Some definitions and results from Nevanlinna theory on annuli}

In this section, we will recall some basic notions of Nevanlinna theory for meromorphic functions on annuli from \cite{LY10} (see also \cite{KK05b}).

For a divisor $\nu$ on $\A (R_0)$, which we may regard as a function on  $\A (R_0)$ with values in $\Z$ whose support is a discrete subset of $\A (R_0),$ and  for a positive integer $M$ (may be $M=+\infty$), we define the counting function of $\nu$ as follows
\begin{align*}
n_0^{[M]}(t)&=\begin{cases}
\sum\limits_{1\le |z|\le t}\min\{M,\nu (z)\}&\text{ if }1\le t<R_0\\
\sum\limits_{t\le |z|<1}\min\{M,\nu (z)\}&\text{ if }\dfrac{1}{R_0}<t< 1
\end{cases}\\
 \text{ and }N_0^{[M]}(r,\nu)&=\int\limits_{\frac{1}{r}}^1 \dfrac {n_0^{[M]}(t)}{t}dt +\int\limits_1^r \dfrac {n_0^{[M]}(t)}{t}dt \quad (1<r<\infty).
\end{align*}
For brevity we set $N_0(r,\nu)=N_0^{[+\infty]}(r,\nu)$ and $\overline N_0(r,\nu)=N_0^{[1]}(r,\nu)$.

For a meromorphic function  $\varphi $, we define
\begin{itemize}
\item $\nu^0_\varphi$ (resp. $\nu^{\infty}_{\varphi}$) the divisor of zeros (resp. divisor of poles) of $\varphi$,
\item $\nu_\varphi=\nu^0_\varphi -\nu^\infty_{\varphi},$
\item $\nu^0_{\varphi,\ge k}=\max\{k,\nu^0_{\varphi}\}$.
\end{itemize}
Note that, by Jensen's formula we have
$$ N_0(r,\nu^0_\varphi)-N_0(r,\nu^\infty_\varphi)=\sum_{t=1/r,r}\dfrac{1}{2\pi}\int\limits_{0}^{2\pi}\log|\varphi (te^{i\theta})| d\theta - \dfrac{1}{\pi}\int\limits_{0}^{2\pi}\log|\varphi(e^{i\theta})| d\theta.$$

Let $f$ be a nonconstant meromorphic function on $\A(R)$. The proximity function of $f$ is defined by
$$ m_0(r,f)=\sum_{t=1/r,r}\dfrac{1}{2\pi}\int\limits_{0}^{2\pi}\log^+|f(te^{i\theta})| d\theta - \dfrac{1}{\pi}\int\limits_{0}^{2\pi}\log^+|f(e^{i\theta})| d\theta $$
and the characteristic function of $f$ is defined by
$$ T_0(r,f)=m_0(r,f)+N_0(r,\nu^\infty_f).$$
Now, we write $f=\frac{f_0}{f_1}$ where $f_0,f_1$ are two holomorphic functions without common zero. It is well-known that
\begin{align*}
T_0(r,f)=\sum_{t=1/r,r}\dfrac{1}{2\pi}\int\limits_{0}^{2\pi}\log\|f(te^{i\theta})\|d\theta+O(1),
\end{align*}
where $\|f\|=\sqrt{|f_0|^2+|f_1|^2}$ and $O(1)$ is a bounded term depending on the choice of $f_0,f_1$.

Throughout this paper, we denote by $S_f(r)$ the quantities satisfying:

$\bullet$ in the case $R_0=+\infty$,
$$ S_f(r)=O(\log (rT_0(r,f)))$$ 
for all $r\in (1,+\infty)$ outside an exceptional set $\Delta_R$ with $\int_{\Delta_R}r^{\lambda -1}dr <+\infty$ for some $\lambda >0$,

$\bullet$ in the case $R_0<+\infty$,
$$ S_f(r)=O\left (\log\biggl (\dfrac{T_0(r,f)}{R_0-r}\biggl )\right )\text{ as }r\longrightarrow R_0$$ 
for all $r\in (1,R_0)$ outside an exceptional set $\Delta'_R$ with $\int_{\Delta'_R}\frac{dr}{(R_0-r)^{\lambda +1}} <+\infty$ for some $\lambda >0$.

Thus, for an admissible meromorphic function $f$ on $\A(R_0)$, we have $S_f(r) = o(T_0(r, f ))$ as $r\longrightarrow R_0$ for all $r\in (1,R_0)$ outside the set $\Delta_R$ or the set $\Delta'_R$ mentioned above, respectively (cf. \cite{KK05b}).

A meromorphic function $a$ on $\A (R_0)$ is said to be small with respect to $f$ if 
$$T_0(r,a)=S_f(r).$$ 

\begin{lemma}[Lemma on logarithmic derivatives \cite{KK05b,LY10}]\label{2.1}
Let $f$ be a nonzero meromorphic function on $\A (R_0)$. Then for each $k\in\mathbb N$ we have
$$ m_0\left (r,\dfrac{f^{(k)}}{f}\right )=S_f(r)\ (1\le r <R_0). $$
\end{lemma}

\begin{theorem}[Second main theorem \cite{KK05b}]\label{2.2}
Let $f$ be a nonconstant meromorphic function on $\A (R_0)$. Let $a_1,\dots ,a_q$ be $q$ distinct values in $\C\cup\{\infty\}$. We have
$$ (q-2)T_0(r,f)\le \sum_{i=1}^q\overline{N}_0(r,\nu^0_{f-a_i})+S_f(r)\ (1\le r< R_0).$$
\end{theorem}

\section{Uniqueness theorem for meromorphic functions sharing finite sets}

Before proving Theorem \ref{1.3}, we first prove Theorem \ref{1.4}. In order to prove Theorem \ref{1.4}, we need the following lemmas.

\begin{lemma}\label{3.1}
 Let $f_1,f_2$ be two meromorphic functions on $\A(R_0)$, $a_1,a_2,a_3$ be three distinct meromorphic functions on $\A(R_0)$ (being not equal to $\infty$) such that $ f_2=(f_1-a_1)/(f_1-a_2).$ Then we have:
\begin{align*}
\mathrm{(a)}\ &T_0(r,f_2)\ge T_0(r,f_1)-\sum_{i=1}^2T_0(r,a_i)+O(1),\\
\mathrm{(b)}\ &\overline{N}_0(r,\nu^0_{f_2})+\overline{N}_0(r,\nu^0_{f_2-1})+\overline{N}_0(r,\nu^\infty_{f_2})\\
& \le \overline{N}_0(r,\nu^\infty_{f_1})+\sum_{i=1}^2\left(\overline{N}_0(r,\nu^0_{f_1-a_i})+2T_0(r,a_i)\right)+O(1).
\end{align*}
Moreover, if we set $b=\dfrac{a_3-a_1}{a_3-a_2}$ then
$$(c)\ \overline{N}_0(r,\nu^0_{f_2-b})\le\overline{N}_0(r,\nu^0_{f_1-a_3})+T_0(r,a_1)+2T_0(r,a_2)+T_0(r,a_3)+O(1).$$
\end{lemma}
\begin{proof} We write $f_i=\frac{f_{i0}}{f_{i1}}$, where $f_{i0}$ and $f_{i1}$ are holomorphic functions defined on $\A(R_0)$ without common zeros, for $i=1,2$. Similarly, we write  $a_i=\frac{a_{i0}}{a_{i1}}$, where $a_{i0}$ and $a_{i1}$ are holomorphic functions defined on $\A(R_0)$ without common zeros, for $i=1,2,3$.  We have
\begin{align*}
f_1=\frac{a_1-a_2f_2}{1-f_2},\text{i.e., }\frac{f_{10}}{f_{11}}=\frac{f_{21}a_{10}a_{21}-f_{20}a_{11}a_{20}}{f_{21}a_{11}a_{21}-f_{20}a_{11}a_{21}}.
\end{align*}
Then, there is a holomorphic function $h$ and a positive constant $C$ such that 
\begin{align*}
&hf_{10}=f_{21}a_{10}a_{21}-f_{20}a_{11}a_{20},\\ 
&hf_{11}=f_{21}a_{11}a_{21}-f_{20}a_{11}a_{21}\\
\text{and }&|h|\cdot\|f_1\|\le C\|f_2\|\cdot\|a_1\|\cdot \|a_2\|.
\end{align*}
This implies that 
$$ T_0(r,f_1)+N_0(r,h)\le T_0(r,f_2)+T_0(r,a_1)+T_0(r,a_2)+O(1).$$
Therefore, 
$$T_0(r,f_2)\ge T_0(r,f_1)-T_0(r,a_1)-T_0(r,a_2)+O(1)$$
and we get the inequality (a).

From the assumption $f_2=\frac{f_1-a_1}{f_1-a_2}$, we easily have
\begin{align*}
\overline{N}_0(r,\nu^0_{f_2})&+\overline{N}_0(r,\nu^0_{f_2-1})+\overline{N}_0(r,\nu^\infty_{f_2})\\
&\le \overline{N}_0(r,\nu^\infty_{f_1})+\sum_{i=1}^2\overline{N}_0(r,\nu^0_{f_1-a_i})+(\overline{N}_0(r,\nu^0_{a_1-a_2})+\overline{N}_0(r,\nu^\infty_{a_1})+\overline{N}_0(r,\nu^\infty_{a_2}))\\
&\le \overline{N}_0(r,\nu^\infty_{f_1})+\sum_{i=1}^2\left(\overline{N}_0(r,\nu^0_{f_1-a_i})+2T_0(r,a_i)\right)+O(1).
\end{align*}
Then we have the inequality (b).

Similarly, we have
$$ \dfrac{1}{f_2-b}=\dfrac{(a_3-a_2)}{(a_1-a_2)}\cdot\left (1+\frac{a_3-a_2}{f_1-a_3}\right).$$
This implies that
$$ \min\{1,\nu^0_{f_2-b}\}\le\min\{1,\nu^\infty_{a_3-a_2}+\nu^0_{a_1-a_2}+\nu^0_{f_1-a_3}\}.$$
It yields that
 $$ \overline{N}_0(r,\nu^0_{f_2-b})\le\overline{N}_0(r,\nu^0_{f_1-a_3})+T_0(r,a_1)+2T_0(r,a_2)+T_0(r,a_3),$$
and we get the inequality (c). This completes the proof of the lemma.
\end{proof}

\begin{lemma}\label{3.2}
 Let $f_1,f_2$ be two meromorphic functions on $\A(R_0)$ and let $a_1,a_2,a_3,a_4$ be four distinct meromorphic functions on $\A(R_0)$ such that
$$ f_2=\frac{f_1-a_1}{f_1-a_2}\cdot\frac{a_3-a_2}{a_3-a_1}.$$
Then we have:
\begin{align*}
\mathrm{(a)}\ &T_0(r,f_2)\ge T_0(r,f_1)-\sum_{i=1}^3T_0(r,a_i)+O(1),\\
\mathrm{(b)}\ &\overline{N}_0(r,\nu^0_{f_2})+\overline{N}_0(r,\nu^0_{f_2-1})+\overline{N}_0(r,\nu^\infty_{f_2})\\
&\hspace{30pt} \le \sum_{i=1}^3\left(\overline{N}_0(r,\nu^0_{f_1-a_i})+3T_0(r,a_i)\right)+O(1).
\end{align*}
Moreover, if we set $b=\dfrac{a_4-a_1}{a_4-a_2}\cdot\dfrac{a_3-a_2}{a_3-a_1}$ then
$$\mathrm{(c)}\  \overline{N}_0(r,\nu^0_{f_2-b})\le\overline{N}_0(r,\nu^0_{f_1-a_4})+2T_0(r,a_1)+3T_0(r,a_2)+2T_0(r,a_3)+T_0(r,a_4).$$
\end{lemma}
\begin{proof} We write $f_i=\frac{f_{i0}}{f_{i1}}$ and  $a_i=\frac{a_{i0}}{a_{i1}}\ (1\le i\le 4)$ similarly as in the proof of Lemma \ref{3.1}. 
By a simple computation, we have
\begin{align*}
&f_1=\frac{a_1(a_3-a_2)-a_2(a_3-a_1)f_2}{(a_3-a_2)-(a_3-a_1)f_2},\\
\text{i.e., }&\frac{f_{10}}{f_{11}}=\frac{f_{21}a_{10}(a_{21}a_{30}-a_{20}a_{31})-f_{20}a_{20}(a_{11}a_{30}-a_{10}a_{31})}{f_{21}a_{11}(a_{21}a_{30}-a_{20}a_{31})-f_{20}a_{21}(a_{11}a_{30}-a_{10}a_{31})}.
\end{align*}
Then, there is a holomorphic function $h$ and a positive constant $C$ such that 
\begin{align*}
&hf_{10}=f_{21}a_{10}(a_{21}a_{30}-a_{20}a_{31})-f_{20}a_{20}(a_{11}a_{30}-a_{10}a_{31}),\\ 
&hf_{11}=f_{21}a_{11}(a_{21}a_{30}-a_{20}a_{31})-f_{20}a_{21}(a_{11}a_{30}-a_{10}a_{31})\\
\text{and }&|h|\cdot\|f_1\|\le C\|f_2\|\cdot\|a_1\|\cdot\|a_2\|\cdot\|a_3\|.
\end{align*}
This implies that 
$$ T_0(r,f_1)+N_0(r,h)\le T_0(r,f_2)+\sum_{i=1}^3T_0(r,a_i)+O(1).$$
Therefore, 
$$T_0(r,f_2)\ge T_0(r,f_1)-\sum_{i=1}^3T_0(r,a_i)+O(1),$$
and we get the inequality (a).

It is also easy to see that
\begin{align*}
\overline{N}_0(r,\nu^0_{f_2})&+\overline{N}_0(r,\nu^0_{f_2-1})+\overline{N}_0(r,\nu^\infty_{f_2}) \\
&\le \sum_{i=1}^3\overline{N}_0(r,\nu^0_{f_1-a_i})+\sum_{1\le i<j\le 3}\overline{N}_0(r,\nu^0_{a_i-a_j})+\sum_{i=1}^3\overline{N}_0(r,\nu^\infty_{a_i})\\
&\le \sum_{i=1}^3\left(\overline{N}_0(r,\nu^0_{f_1-a_i})+3T_0(r,a_i)\right)+O(1).
\end{align*}
Then we have the inequality (b).

On the other hand, we have
$$ \dfrac{1}{f_2-b}=\dfrac{(a_3-a_1)(a_4-a_2)}{(a_3-a_2)(a_1-a_2)}\cdot\left (1+\frac{a_4-a_2}{f_1-a_4}\right).$$
This implies that
$$ \min\{1,\nu^0_{f_2-b}\}\le\min\{1,\nu^\infty_{a_3-a_1}+\nu^\infty_{a_4-a_2}+\nu^0_{a_3-a_2}+\nu^0_{a_1-a_2}+\nu^0_{f_1-a_4}\}.$$
It yields that
 $$ \overline{N}_0(r,\nu^0_{f_2-b})\le\overline{N}_0(r,\nu^0_{f_1-a_4})+2T_0(r,a_1)+3T_0(r,a_2)+2T_0(r,a_3)+T_0(r,a_4).$$
This completes the proof of the lemma.
\end{proof}

The proof of Theorem \ref{1.4} is a straightforward deduction from the next lemma.

\begin{lemma}\label{3.3}
Let $g$ be a nonconstant meromorphic function on $\A (R_0)$. Let $a_1,\dots ,a_5$ be five distinct meromorphic functions (may be equal to $\infty$). We have
$$2T_0(r,g)\le \sum_{i=1}^5\overline N_0(r,\nu^0_{g-a_i})+35\sum_{i=1}^5T_0(r,a_i)+S(r)\ (1\le r< R_0),$$
where $S(r)=S_g(r)+\sum_{i=1}^5S_{a_i}(r)$.
\end{lemma}
Then, in order to prove Theorem \ref{1.4}, it is sufficient to prove Lemma \ref{3.3}.
\begin{proof}

We split the proof of the lemma into two parts.

\noindent
\textbf{Part A.} We first consider the case where $a_i\not\equiv\infty$ for all $i=1,\ldots,5$. We set
$$f = \frac{g-a_2}{g-a_1} \cdot \frac{a_3-a_1}{a_3-a_2}, b_1= \frac{a_4-a_2}{a_4-a_1} \cdot \frac{a_3-a_1}{a_3-a_2}, b_2= \frac{a_5-a_2}{a_5-a_1} \cdot \frac{a_3-a_1}{a_3-a_2},b_3=0\text{ and }b_4=1.$$
By Lemma \ref{3.2}, we have
\begin{align}\label{3.4}
T_0(r,g)&\le T_0(r,f)+\sum_{i=1}^3T_0(r,a_i)+O(1),\\
\label{3.5}
T_0(r,b_1)&\le\sum_{i=1}^4T_0(r,a_i)+O(1),\\
\label{3.6}
T_0(r,b_2)&\le\sum_{i=1}^3T_0(r,a_i)+T_0(r,a_5)+O(1)
\end{align}
and
\begin{align}\label{3.7}
\begin{split}
\overline{N}_0(r,\nu^\infty_f)+\sum_{i=1}^4\overline{N}_0(r,\nu^0_{f-b_i})
&\le \sum_{i=1}^5\overline{N}_0(r,\nu^0_{g-a_i})+7T_0(r,a_1)+9T_0(r,a_2)\\
&+7T_0(r,a_3)+T_0(r,a_4)+T_0(r,a_5)+O(1).
\end{split}
\end{align}

We now prove the following claim.
\begin{claim}\label{3.8}
$2T_0(r,f)\le \overline{N}_0(r,\nu^\infty_f)+\sum_{i=1}^4\overline{N}_0(r,\nu^0_{f-b_i})+18\sum_{i=1,2}T_0(r,b_i)+S(r).$
\end{claim}
Indeed, if $b_1$ or $b_{2}$ is constant then the claim directly follows from Theorem \ref{2.2}.
Therefore, we may assume that both $b_1$ and $b_2$ are not constant. We define
\begin{equation}\label{3.9}
F= \begin{vmatrix}
ff'&f'&f^2-f \\
b_1b_1'&b_1'&b_1^2-b_1 \\
b_2b_2' &b_2'&b_2^2-b_2
\end{vmatrix}.
\end{equation}
Consider the following two cases.

\textbf{Case 1: }$F(z)\equiv0$. From (\ref{3.9}) we get
\begin{align}\label{3.10}
\biggl (\dfrac{b_1'}{b_1}-\dfrac{b_2'}{b_2}\biggl )\biggl (\dfrac{f'}{f-1}-\dfrac{b_2'}{b_2-1}\biggl )\equiv \biggl (\dfrac{b_1'}{b_1-1}-\dfrac{b_2'}{b_2-1}\biggl )\biggl (\dfrac{f'}{f}-\dfrac{b_2'}{b_2}\biggl ).
\end{align}
We distinguish the following four subcases. 

\emph{Subcase 1:} $\dfrac{b_1'}{b_1}\equiv\dfrac{b_2'}{b_2}.$ We have $\dfrac{b_1'}{b_1-1}\not\equiv\dfrac{b_2'}{b_2-1}$, since otherwise $b_1$ and $b_{2}$ are constants. Therefore, from (\ref{3.10}), we have $\dfrac{f'}{f}\equiv\dfrac{b_2'}{b_2}$, and hence $f=cb_2$ with a constant $c$. This implies that $T_0(r,f)=T_0(r,b_2)$ and the claim is proved in this subcase.

\emph{Subcase 2:} $\dfrac{b_1'}{b_1-1}\equiv\dfrac{b_2'}{b_2-1}.$ By the same arguments as in Subcase 1, we get again the inequality of the claim in this subcase. 

\emph{Subcase 3:} $\dfrac{b_1'}{b_1}\not\equiv\dfrac{b_2'}{b_2}, \dfrac{b_1'}{b_1-1}\not\equiv\dfrac{b_2'}{b_2-1},\dfrac{b_1'}{b_1}-\dfrac{b_2'}{b_2} \equiv\dfrac{b_1'}{b_1-1}-\dfrac{b_2'}{b_2-1}.$ The identity (\ref{3.10}) implies that  
$$\dfrac{f'}{f-1}-\dfrac{f'}{f}\equiv \dfrac{b_2'}{b_2-1}-\dfrac{b_2'}{b_2}.$$
Then, we have
$$\dfrac{f-1}{f}\equiv c \cdot\dfrac{b_2-1}{b_2},$$
where $c$ is a constant. Therefore $f=\frac{b_2}{(c-1)b_2-c}$, and hence $T_0(r,f)=T_0(r,b_2)$. We get again the inequality of the claim in this subcase.

\emph{Subcase 4:} $\dfrac{b_1'}{b_1}\not\equiv\dfrac{b_2'}{b_2}, \dfrac{b_1'}{b_1-1}\not\equiv\dfrac{b_2'}{b_2-1},\dfrac{b_1'}{b_1}-\dfrac{b_2'}{b_2} \not\equiv\dfrac{b_1'}{b_1-1}-\dfrac{b_2'}{b_2-1}.$ The identity (\ref{3.10}) may be rewritten as
\begin{align}\label{3.11}
\biggl (\dfrac{b_1'}{b_1}-\dfrac{b_2'}{b_2}\biggl )\dfrac{f'}{f-1}-\biggl (\dfrac{b_1'}{b_1-1}-\dfrac{b_2'}{b_2-1}\biggl )\dfrac{f'}{f}\equiv \dfrac{b_1'b_2'}{b_1(b_2-1)}-\dfrac{b_2'b_1'}{b_2(b_1-1)}.
\end{align}
From (\ref{3.11}), we see that each zero of $(f-1)$ must be a zero or an $1$-point or a pole of $b_j\  (j=1, 2)$ or a zero of $\dfrac{b_1'}{b_1}-\dfrac{b_2'}{b_2}$. Therefore,
\begin{align}\label{3.12}
\min\{1,\nu^0_{f-1}\}\le\sum_{i=1,2}\sum_{a=0,1,\infty}\min\{1,\nu^{0}_{b_i-a}\}+\min\left\{1,\nu^0_{\frac{b_1'}{b_1}-\frac{b_2'}{b_2}}\right\} .
\end{align}
Similarly, we have
\begin{align}\label{3.13}
\min\{1,\nu^0_{f}\}\le\sum_{i=1,2}\sum_{a=0,1,\infty}\min\{1,\nu^{0}_{b_i-a}\}+\min\left\{1,\nu^0_{\frac{b_1'}{b_1-1}-\frac{b_2'}{b_2-1}}\right\} .
\end{align}
From (\ref{3.11}), we also see that each pole of $f$ must be a zero or an $1-$point or a pole of $b_j\  (j=1, 2)$ or a zero of $\biggl (\dfrac{b_1'}{b_1}-\dfrac{b_2'}{b_2}\biggl )- \biggl (\dfrac{b_1'}{b_1-1}-\dfrac{b_2'}{b_2-1}\biggl )$. Therefore, by the similar arguments, we have
\begin{align}\label{3.14}
\min\{1,\nu^\infty_{f}\}\le\sum_{i=1,2}\sum_{a=0,1,\infty}\min\{1,\nu^{0}_{b_i-a}\}+\min\left\{1,\nu^0_{\bigl (\frac{b_1'}{b_1}-\frac{b_2'}{b_2}\bigl )- \bigl (\frac{b_1'}{b_1-1}-\frac{b_2'}{b_2-1}\bigl )}\right\}.
\end{align}
Combining (\ref{3.12}), (\ref{3.13}) and (\ref{3.14}), we have 
\begin{align}\label{3.15}
\begin{split}
\sum_{a=0,1,\infty}\min\{1,\nu^0_{f-a}\}&\le\sum_{i=1,2}\sum_{a=0,1,\infty}\min\{1,\nu^{0}_{b_i-a}\}+\min\bigl\{1,\nu^0_{\frac{b_1'}{b_1}-\frac{b_2'}{b_2}}\bigl\}\\
&+\min\bigl\{1,\nu^0_{\frac{b_1'}{b_1-1}-\frac{b_2'}{b_2-1}}\bigl\}+\min\bigl\{1,\nu^0_{\bigl (\frac{b_1'}{b_1}-\frac{b_2'}{b_2}\bigl )- \bigl (\frac{b_1'}{b_1-1}-\frac{b_2'}{b_2-1}\bigl )}\bigl\}.
\end{split}
\end{align}
By Theorem \ref{2.2}, we get
\begin{align*}
T_0( r,f)&\le \overline{N}_0(r,\nu^0_f)+\overline{N}_0(r,\nu^0_{f-1})+\overline{N}_0(r,\nu^\infty_f)+S(r)\\
&\le 2\sum_{i=1}^2\left(T_0\left(r,\frac{b_i'}{b_i}\right)+T_0\left(r,\frac{b_i'}{b_i-1}\right)\right)+S(r)\\
&\le 2\sum_{i=1}^2\left(N_0\left(r,\nu^\infty_{\frac{b_i'}{b_i}}\right)+N_0\left(r,\nu^\infty_{\frac{b_i'}{b_i-1}}\right)\right)+S(r)\\
&\le 2\sum_{i=1}^2(N_0(r,\nu^0_{b_i})+2N_0(r,\nu^\infty_{b_i})+N_0(r,\nu^0_{b_i-1}))\\
&\le 8T_0(r,b_1)+8T_0(r,b_2).
\end{align*}
Then, we have the desired inequality of the claim in this subcase.

\textbf{Case 2:} $F(z)\not\equiv0$. For $t=r$ or $t=\frac{1}{r}\ (r>1)$, we set
\begin{align*}
\delta (z)&= \min\{1,|b_1(z)|,|b_2(z)|,|b_1(z)-1|,|b_2(z)-1|,|b_1(z)-b_2(z)|\},\\
\theta _j (t)& = \{\theta:|f(te^{i\theta})-b_j(te^{i\theta})|\le\delta (te^{i\theta})\},(j=1,2),\\
\theta _j(t)& = \{\theta:|f(te^{i\theta})-b_j|\le\delta (te^{i\theta})\},(j=3,4).
\end{align*}
It is clear that the sets $\theta_i(t)\cap\theta_j(t) \ ( i\neq j, i, j=1, 2, 3, 4)$ have at most finitely many points. Then, we easily see that  
\begin{align}\label{3.16}
\begin{split}
\sum_{t=1/r,r}\dfrac{1}{2\pi}\int\limits_{0}^{2\pi}\log \dfrac{1}{\delta (te^{i\theta})}d\theta
&\le m_0\left( r,\dfrac{1}{b_1}\right)+m_0\left( r,\dfrac{1}{b_2}\right)+m_0\left( r,\dfrac{1}{b_1-1}\right)\\
&+m_0\left( r,\dfrac{1}{b_2-1}\right)+m_0\left( r,\dfrac{1}{b_1-b_2}\right)\\
&\le T(r)+S(r)+O(1),
\end{split}
\end{align}
where $T(r)=3 T_0(r,b_1)+3T_0(r,b_2).$ We also see that
\begin{align*}
ff'&=(f-b_1)(f'-b_1')+b_1'(f-b_1)+b_1(f'-b_1')+b_1b_1',\\
f'&=(f'-b_1')+b_1',\\
f^2-f&=(f-b_1)^2+(2b_1-1)(f-b_1)+b_1^2-b_1.
\end{align*}
Substituting these functions (on the right hand side) into (\ref{3.9}), by the properties of the Wronskian, we have
\begin{equation}\label{3.17}
F= \begin{vmatrix}
\varphi&f'-b_1'&\psi \\
b_1b_1'&b_1'&b_1'-b_1\\
b_2b_2' &b_2'&b_2'-b_2
\end{vmatrix}, 
\end{equation}
where
\begin{align*}
\varphi &=(f-b_1)(f'-b_1')+b_1'(f-b_1)+b_1(f'-b_1'),\\
\psi&=(f-b_1)^2+(2b_1-1)(f-b_1).
\end{align*}
From (\ref{3.17}) we see that each zero with multiplicity $p$ $(p>1)$ of $f-b_1,$ which is neither a pole of $b_1$ nor a pole of $b_2,$ must be a zero of  $F$ with multiplicity at least $p-1$. Similarly, each zero with multiplicity $p$ $(p>1)$ of $f-b_2,$ which is neither a pole of $b_1$ nor a pole of $b_2,$ must be a zero of $F$ with multiplicity at least $p-1$. Moreover, from (\ref{3.9}) one sees that each zero of multiplicity $p$ $(p>1)$ of $f$ or $f-1,$ which is neither a pole of  $b_1$ nor a pole of $b_2,$ must be a zero of $F$ with multiplicity at least $p-1$. This implies that
\begin{align}\label{3.18}
\sum_{i=1}^4(N(r,\nu^0_{f-b_i})-\overline{N}_0(r,\nu^0_{f-b_i}))\le N(r,\nu^0_F).
\end{align}

We now estimate the quantities $m_0(r,\frac{1}{f-b_j})\ (1\le j\le 4)$.

Since $|f(te^{i\theta})-b_1(te^{i\theta})|\le\delta(te^{i\theta})\le 1$ for every $\theta\in\theta_1(t)$, 
\begin{align}\label{3.19}
\begin{split}
&\sum_{t=1/r,r}\frac{1}{2\pi }\int\limits_{\theta_1(t)}\log^+\left| \dfrac{F}{f-b_1}\right|d\theta\le\sum_{t=1/r,r} \frac{1}{2\pi }\int\limits_{\theta_1(t)}\log\left (1+\left| \dfrac{F}{f-b_1}\right|\right)d\theta\\
&\le \sum_{t=1/r,r}\frac{1}{2\pi }\int\limits_{\theta_1(t)}\log\left\{(1+|b_1'|)(1+|b_1|)\left(1+\left|\frac{f'-b_1'}{f-b_1}\right|\right)\right\}d\theta\\
&\hspace{30pt}+\sum_{t=1/r,r}\frac{1}{2\pi }\int\limits_{\theta_1(t)}\log\left\{(1+|b_1|)(1+|b_1'|)\right\}d\theta\\
&\hspace{30pt}+\sum_{t=1/r,r}\frac{1}{2\pi }\int\limits_{\theta_1(t)}\log\left\{(1+|b_2|)(1+|b_2'|)\right\}d\theta+O(1)\\
&\le m_0\left( r,\dfrac{f'-b_1'}{f-b_1}\right)+2m_0\left(r,\frac{b_1'}{b_1}\right)+m_0\left(r,\frac{b_2'}{b_2}\right)\\
&\hspace{30pt}+4\sum_{t=1/r,r}\frac{1}{2\pi }\int\limits_{\theta_1(t)}\log^+|b_1|d\theta+2\sum_{t=1/r,r}\frac{1}{2\pi }\int\limits_{\theta_1(t)}\log^+|b_2|d\theta+O(1)\\
&\le 4\sum_{t=1/r,r}\frac{1}{2\pi }\int\limits_{\theta_1(t)}\log^+|b_1|d\theta+2\sum_{t=1/r,r}\frac{1}{2\pi }\int\limits_{\theta_1(t)}\log^+|b_2|d\theta+S(r).
\end{split}
\end{align}
Here, the second inequality comes from the fact that
\begin{align}\label{3.20}
 \log^+|\det(x_{ij};1\le i,j\le 3)|\le \sum_{i=1}^3\log^+\max\{|x_{ij}|;1\le j\le 3\}+O(1)
\end{align}
for every $(3\times 3)$-matrix of complex numbers $(x_{ij})_{1\le i,j\le 3}$.

Therefore, from (\ref{3.16}) and (\ref{3.19}) we have
\begin{align}\label{3.21}
\begin{split}
&m_0\left( r,\dfrac{1}{f-b_1}\right)\le \sum_{t=1/r,r}\frac{1}{2\pi }\int\limits_{\theta_1(t)}\log^+\left| \dfrac{1}{f-b_1} \right|d\theta +\dfrac{1}{2\pi}\int\limits_{0}^{2\pi}\log \dfrac{1}{\delta(re^{i\theta})}d\theta\\
&\le \sum_{t=1/r,r}\frac{1}{2\pi }\int\limits_{\theta_1(t)}\log^+\left| \dfrac{F}{f-b_1} \right|d\theta +\frac{1}{2\pi } \int\limits_{\theta_1(r)}\log^+\dfrac{1}{|F|}d\theta+T(r)+S(r)\\
&= \sum_{t=1/r,r}\frac{1}{2\pi }\int\limits_{\theta_1(t)}\biggl(\log^+ \dfrac{1}{|F|}d\theta+4\log^+|b_1|+2\log^+|b_2|\biggl)d\theta+T(r)+S(r). 
\end{split}
\end{align}
Similarly, we get
\begin{align}
\label{3.22}
\begin{split}
m_0\left( r,\dfrac{1}{f-b_2}\right)&\le \sum_{t=1/r,r}\frac{1}{2\pi }\int\limits_{\theta_2(t)}\biggl(\log^+ \dfrac{1}{|F|}d\theta+2\log^+|b_1|+4\log^+|b_2|\biggl)d\theta\\
&+T(r)+S(r).
\end{split}
\end{align}

On the other hand, for each $\theta\in\theta_j(t)\ (j=3,4)$ we have $|f(te^{i\theta})-b_j|\le2$ and hence
$$ \log^+\frac{|F(te^{i\theta})|}{|f(te^{i\theta})-b_j|}\le \log^+\frac{|f'(te^{i\theta})|}{|f(te^{i\theta})-b_j|}+\sum_{i=1}^2\left(2\log^+|b_i(te^{i\theta})|+\log^+\frac{|b_i'(te^{i\theta})|}{|b_i(te^{i\theta})|}\right)+O(1).$$
Similarly as above, for $j=3,4$, we get
\begin{align}
\label{3.23}
\begin{split}
 m_0\left( r,\dfrac{1}{f-b_j}\right)&\le  \sum_{t=1/r,r}\frac{1}{2\pi }\int\limits_{\theta_j(t)}\biggl(\log^+ \dfrac{1}{|F|}d\theta+2\log^+|b_1|+2\log^+|b_2|\biggl)d\theta\\
&+T(r)+S(r).
\end{split}
\end{align}

Combining (\ref{3.21}), (\ref{3.22}) and (\ref{3.23}), we have
$$\sum_{i=2}^5m_0\left( r,\dfrac{1}{f-a_i}\right)\le m_0\left( r,\dfrac{1}{F}\right)+4m_0(r,b_1)+4m_0(r,b_2)+4T(r)+S(r).$$
Therefore,
\begin{align*}
4T_0( r,f)&\le N_0(r,\nu^0_f)+N_0(r,\nu^0_{f-1})+N_0(r,\nu^0_{f-b_1})+N_0(r,\nu^0_{f-b_2})-N(r,\nu^0_F)\\
&+T_0(r,F)+4m_0(r,b_1)+4m_0(r,b_2)+4T(r)+S(r).
\end{align*}
Combining this inequality and (\ref{3.18}), we obtain
\begin{align}\label{3.24}
\begin{split}
4T_0( r,f)&\le \overline N_0(r,\nu^0_f)+\overline N_0(r,\nu^0_{f-1})+\overline N_0(r,\nu^0_{f-b_1})+\overline N_0(r,\nu^0_{f-b_2})\\
&+T_0(r,F)+4m_0(r,b_1)+4m_0(r,b_2)+4T(r)+S(r).
\end{split}
\end{align}
Also, from (\ref{3.9}) and (\ref{3.20}), we have
\begin{align*}
m_0(r,F)&\le 2m_0(r,f)+m_0\left(r,\frac{f'}{f}\right)+\sum_{i=1}^2\left(2m_0(r,b_i)+m_0\left(r,\frac{b_1'}{b_1}\right)\right)+S(r)\\
&\le 2m_0(r,f)+2m_0(r,b_1)+2m_0(r,b_2)+S(r)
\end{align*}
and
$$ N_0(r,\nu^\infty_F)\le 2N_0(r,\nu^\infty_f)+\overline{N}_0(r,\nu^\infty_f)+ 3\sum_{i=1}^2N_0(r,\nu^\infty_{b_i}).$$
These inequalities imply that
\begin{align}\label{3.25}
T_0(r,F)\le 2T_0(r,f)+\overline{N}_0(r,\nu^\infty_f)+\sum_{i=1}^2N_0(r,\nu^\infty_{b_i})+\frac{2}{3}T(r)+S(r).
\end{align}
From (\ref{3.24}) and (\ref{3.25}), we have 
\begin{align*}
2T_0( r,f)\le& \overline N_0(r,\nu^0_f)+\overline N_0(r,\nu^0_{f-1})+\overline N_0(r,\nu^\infty_{f})+\overline N_0(r,\nu^0_{f-b_1})\\
&+\overline N_0(r,\nu^0_{f-b_2})+18T_0(r,b_1)+18T_0(r,b_2)+S(r).
\end{align*}
Then, the claim is proved.

Now, we return to Part A of the proof of the lemma. From (\ref{3.4})-(\ref{3.7}) and Claim \ref{3.8}, we get
$$2T_0(r,g)\le \sum_{i=1}^5\overline{N}_0(r,\nu^0_{g-a_i})+45\sum_{i=1,3}T_0(r,a_i)+47T_0(r,a_2)+19\sum_{i=4,5}T_0(r,a_i)+S(r).$$
Then, by the same arguments we have
$$2T_0(r,g)\le \sum_{i=1}^5\overline{N}_0(r,\nu^0_{g-a_i})+45\sum_{j=1,3}T_0(r,a_{i_j})+47T_0(r,a_{i_2})+19\sum_{j=4,5}T_0(r,a_{i_j})+S(r)$$
for any permutation $(i_1,\ldots,i_5)$ of $\{1,\ldots,5\}$. Summing-up both sides of the above inequalities over all such permutations, we obtain
\begin{align}\label{new1}
2T_0(r,g)\le \sum_{i=1}^5\overline{N}_0(r,\nu^0_{g-a_i})+35\sum_{i=1}^5T_0(r,a_i)+S(r).
\end{align}
Hence the lemma is proved for the case where all $a_i\not\equiv\infty$ for $i=1,\ldots,5$.

\noindent
\textbf{Part B.} Now, we consider the remaining case where there is a function among $\{a_1,\ldots,a_5\}$ equal to $\infty$, for instance $a_1\equiv\infty$. We set
$$ h=\frac{g-a_2}{g-a_3}, c_1=\frac{a_4-a_2}{a_4-a_3}, c_2=\frac{a_5-a_2}{a_5-a_3},c_3=0\text{ and }c_4=1.$$
By Lemma \ref{3.1}, we have
\begin{align}\label{3.26}
T_0(r,g)&\le T_0(r,h)+\sum_{i=2}^3T_0(r,a_i)+O(1),\\
\label{3.27}
T_0(r,c_1)&\le\sum_{i=2}^4T_0(r,a_i)+O(1),\\
\label{3.28}
T_0(r,c_2)&\le\sum_{i=2}^3T_0(r,a_i)+T_0(r,a_5)+O(1)
\end{align}
and
\begin{align}\label{3.29}
\begin{split}
\overline{N}_0(r,\nu^0_h)&+\overline{N}_0(r,\nu^\infty_h)+\overline{N}_0(r,\nu^0_{h-1})+\overline{N}_0(r,\nu^0_{h-c_1})+\overline{N}_0(r,\nu^0_{f-c_2})\\
&\le \sum_{i=1}^5\overline{N}_0(r,\nu^0_{g-a_i})+4T_0(r,a_2)+8T_0(r,a_3)\\
&+T_0(r,a_4)+T_0(r,a_5)+O(1).
\end{split}
\end{align}
Applying Claim \ref{3.8} for functions $h,c_1,c_2,c_3,c_4$, we get
$$2T_0(r,h)\le \overline{N}_0(r,\nu^\infty_{h})+\sum_{i=1}^4\overline{N}_0(r,\nu^0_{h-c_i})+18\sum_{i=1}^2T_0(r,c_i)+S(r).$$
Combining (\ref{3.26})-(\ref{3.29}) and the above inequality, we get
$$2T_0(r,g)\le\sum_{i=1}^5\overline{N}_0(r,\nu^0_{g-a_i})+42T(r,a_2)+46T(r,a_3)+19\sum_{i=4,5}T(r,a_i)+S(r).$$
Then, by the same arguments, we have
$$2T_0(r,g)\le \sum_{i=1}^5\overline{N}_0(r,\nu^0_{g-a_i})+42T(r,a_{i_2})+46T(r,a_{i_3})+19\sum_{j=4,5}T(r,a_{i_j})+S(r)$$
for any permutation $(i_2,\ldots,i_5)$ of $\{2,\ldots,5\}$. Summing-up both sides of the above inequalities over all such permutations, we obtain
\begin{align}\label{new2}
2T_0(r,g)\le \sum_{i=1}^5\overline{N}_0(r,\nu^0_{g-a_i})+\frac{33}{2}\sum_{i=2}^5T_0(r,a_i)+S(r).
\end{align}
This completes the proof of the lemma for the case where there is at least one function $a_i\in\{a_1,\ldots,a_5\}$ equal to $\infty$.

From the inequality (\ref{new1}) of Part A and the inequality (\ref{new2}) of Part B, we have the proof of the lemma.
\end{proof} 

We now prove the main result of the paper.

\begin{proof}[{\sc Proof of Theorem \ref{1.3}}]
Suppose contrarily that there exists an admissible meromorphic function $g$ on $\A(R_0)$ and mutually distinct admissible meromorphic functions $f_i\ (1\le i\le 5)$ on $\A(R_0)$ such that $f_i$ and $g$ share the set $S$ with multiplicity truncated to level $\ell$, where $g=f_1$. We note that
$$ P_S(w)=(w-a_1)\cdots (w-a_q).$$
We set $\Psi_i:= P_S(f_i)/P_S(g)\ (1\le i\le 5).$

\begin{claim}\label{3.30}
 One has $T_0(r,f_j)=O(T_0(r,g))+S_{f_j}(r)$ and $T_0(r,g)=O(T_0(r,f_j))+S_g(r)$, in particular $S_{f_j}(r)=S_g(r)$ for every $1\le j\le 5$.
\end{claim}
Indeed, by Theorem \ref{2.2} we have
\begin{align*}
(q-2)T_0(r,g)&\le\sum_{i=1}^q\overline{N}_0(r,\nu^0_{g-a_i})+S_g(r)\\ 
&=\sum_{j=1}^q\overline{N}_0(r,\nu^0_{f_j-a_i})+S_g(r)\\
&\le qT_0(r,f_j)+S_g(r).
\end{align*}
Then, $T_0(r,g)=O(T_0(r,f_j))+S_g(r)$. Similarly $T_0(r,f_j)=O(T_0(r,g))+S_{f_j}(r)$ and hence $S_{f_j}(r)=S_g(r)\ (1\le j\le 5)$.

We set $\varphi_j=P'_S(f_j)f_j'/P_S(f_j)$ and $\varphi =\varphi_1$. By the definition of $\Psi_j$, we have
$$\varphi=\alpha_j+\varphi_j,\text{ where }\alpha_j=\frac{\Psi_j'}{\Psi_j}\ (1\le j\le 5).$$
By the lemma on logarithmic derivative, we have
$$ m_0(r,\varphi_j)=S_{P_S(f_j)}(r)=S_g(r)\ \forall 1\le j\le 5\ (\text{note that }T_0(r,P_S(f_j))=qT_0(r,f_j)).$$
Since each pole of $\varphi_j$ is a simple pole and must be either zero or pole of $P_S(f_j)$, we have
$$ N_0(r,\nu^\infty_{\varphi_j})=\overline{N}_0(r,\nu^0_{P_S(f_j)})+\overline{N}_0(r,\nu^\infty_{P_S(f_j)}) \le 2q T_0(r,f_j).$$
It yields that
$$T_0(r,\varphi_j)=m_0(r,\varphi_j)+N_0(r,\nu^\infty_{\varphi_j})\le 2q T_0(r,f_j)+S_g(r).$$
In particular $T_0(r,\varphi_j)=O(T_0(r,g))$. Moreover, by Theorem \ref{2.2} we have
\begin{align*}
N_0(r,\nu^\infty_{\varphi})&\ge \overline{N}_0(r,\nu^0_{P_S(g)})+\overline{N}_0(r,\nu^\infty_{P_S(g)})\\
&=\overline{N}_0(r,\nu^0_{P_S(f_j)})+\overline{N}_0(r,\nu^\infty_{P_S(f_j)})\\
&=\sum_{i=1}^q\overline{N}_0(r,\nu^0_{f_j-a_i})+q\overline{N}_0(r,\nu^\infty_{f_j})\\
&\ge (q-1)(T_0(r,f_j)+\overline{N}_0(r,\nu^\infty_{f_j}))+S_{f_j}(r).
\end{align*}
This implies that
\begin{align}\label{3.31}
(q-1)T_0(r,f_j)\le T_0(r,\varphi)-(q-1)\overline{N}_0(r,\nu^\infty_{f_j})+S_g(r).
\end{align}

Also, we note that
\begin{align*}
T_0(r,\alpha_j)&=m_0(r,\alpha_j)+N_0(r,\nu^\infty_{\alpha_j})\\
&\le m(r,\varphi)+m(r,\varphi_j)+\overline{N}_0(r,\nu^0_{\Psi_j})+\overline{N}_0(r,\nu^\infty_{\Psi_j})+O(1)\\
&=\overline{N}_0(r,\nu^0_{P_S(g)/P_S(f_j)})+\overline{N}_0(r,\nu^\infty_{P_S(g)/P_S(f_j)})+S_g(r)\\
&\le\sum_{i=1}^q\overline{N}_0(r,\nu^0_{f_j-a_i,\ge\ell})+S_g(r)\le \frac{q}{\ell}T_0(r,f_j)+S_g(r).
\end{align*}
This also yields that $T_0(r,\alpha_j)=O(T_0(r,g))$.

We now show that $\alpha_1,\ldots,\alpha_5$ are mutually distinct. Indeed, suppose contrarily that $\alpha_i=\alpha_j$  for some $i\ne j$. Then, there exists a nonzero constant $c_0$ with $\Psi_i=c_0\Psi_j$ and hence $c_0P_S(f_i)=P_S(f_j).$ This contradicts the assumption that $P_S(w)$ is a uniqueness polynomial for admissible meromorphic functions on $\A(R_0)$ and $f_i\ne f_j$.

Applying Theorem \ref{1.4} to the function $\varphi$ and functions $\alpha_1,\ldots,\alpha_5$, we obtain
$$2T_0(r,\varphi)\le\sum_{j=1}^5\overline{N}_0(r,\nu^0_{\varphi-\alpha_j})+35\sum_{j=1}^5T_0(r,\alpha_j)+S_g(r).$$
On the other hand, we have
$$ \overline{N}_0(r,\nu^0_{\varphi-\alpha_j})=\overline{N}_0(r,\nu^0_{\varphi_j})\le\overline{N}_0(r,\nu^0_{f_j'})+\sum_{t=1}^k\overline{N}_0(r,\nu^0_{f_j-e_t}), $$
where $e_1,\ldots,e_k$ are all of distinct zeros of $P'_S(w)$. Since $\overline{N}_0(r,\nu^0_{f_j-e_t})\le T_0(r,f_j)+O(1)$ and
\begin{align*}
\overline{N}_0(r,\nu^0_{f_j'})\le&T_0(r,f_j')+O(1)=m_0(r,f_j')+N_0(r,\nu^\infty_{f_j'})+O(1)\\ 
\le&m_0(r,f_j)+m_0\left(r,\frac{f_j'}{f_j}\right)+N_0(r,\nu^\infty_{f_j})+\overline{N}_0(r,\nu^\infty_{f_j})+O(1)\\
\le&T_0(r,f_j)+\overline{N}_0(r,\nu^\infty_{f_j})+S_g(r),
\end{align*}
we have
\begin{align}\label{3.32}
\sum_{j=1}^5\overline{N}_0(r,\nu^0_{\varphi-\alpha_j})\le (k+1)\sum_{j=1}^5T_0(r,f_j)+\sum_{j=1}^5\overline{N}_0(r,\nu^\infty_{f_j})+S_g(r).
\end{align}
Then, combining (\ref{3.31}) and (\ref{3.32}) we have
\begin{align*}
2(q-1)T_0(r,f_i)&\le 2T(r,\varphi)-2(q-1)\overline{N}_0(r,\nu^\infty_{f_i})+S_g(r)\\
&\le \left (k+1+\frac{35q}{\ell}\right)\sum_{j=1}^5T_0(r,f_j)+\sum_{j=1}^5\overline{N}_0(r,\nu^\infty_{f_j})\\
&\ \ \ -2(q-1)\overline{N}_0(r,\nu^\infty_{f_i})+S_g(r)\ \forall i=1,\ldots,5.
\end{align*}
Summing up this inequality over all $i=1,\ldots, 5$, we easily obtain
$$2(q-1)\sum_{i=1}^5T_0(r,f_i) \le 5\left (k+1+\frac{35q}{\ell}\right)\sum_{j=1}^5T_0(r,f_j)+S_g(r).$$
Letting $r\rightarrow R_0$ (outside the exceptional set), we get
$$2(q-1)\le 5\left (k+1+\frac{35q}{\ell}\right),\text{ i.e., }q\le\frac{(5k+7)\ell}{2\ell-175}.$$
This contradiction completes the proof of the theorem.
\end{proof}

\section*{Disclosure statement}
No potential conflict of interest was reported by the author(s).

\vskip0.2cm
{\footnotesize 
\noindent
{\sc Si Duc Quang}
\vskip0.05cm
\noindent
Department of Mathematics, Hanoi National University of Education,\\
136-Xuan Thuy, Cau Giay, Hanoi, Vietnam.
\vskip0.05cm
\noindent
\textit{E-mail}: quangsd@hnue.edu.vn

\end{document}